\documentclass[11pt, reqno]{amsart}
\usepackage{lmodern}
\usepackage{amsmath, amsthm, amssymb, amsfonts}
\usepackage[normalem]{ulem}
\usepackage{hyperref}
\usepackage{multirow}
\usepackage{bm}
\usepackage{extarrows}
\usepackage{geometry}

\usepackage{verbatim}
\usepackage{caption}
\usepackage{mathtools}

\usepackage{tikz-cd}

\theoremstyle{plain}
\newtheorem{thm}{Theorem}[section]
\newtheorem{cor}[thm]{Corollary}

\newtheorem{prop}[thm]{Proposition}
\newtheorem{conj}[thm]{Conjecture}

\theoremstyle{definition}
\newtheorem{defn}[thm]{Definition}

\theoremstyle{remark}
\newtheorem{rmk}[thm]{Remark}

\newcommand{\BC}{{\mathbb{C}}}
\newcommand{\BD}{{\mathbb{D}}}

\newcommand{\BH}{{\mathbb{H}}}

\newcommand{\BQ}{{\mathbb{Q}}}

\newcommand{\CC}{{\mathcal C}}
\newcommand{\CD}{{\mathcal D}}

\newcommand{\CP}{{\mathcal P}}

\newcommand{\Fg}{{\mathfrak{g}}}

\newcommand{\Fp}{{\mathfrak{p}}}

\DeclareFontFamily{OT1}{rsfs}{}
\DeclareFontShape{OT1}{rsfs}{n}{it}{<-> rsfs10}{}
\DeclareMathAlphabet{\curly}{OT1}{rsfs}{n}{it}

\begin{document}
\title{Local multiplicativity of perverse filtrations}
\date{\today}

\author{Zili Zhang}
\address{School of Mathematical Sciences,
Key Laboratory of Intelligent Computing and Applications (Ministry of Education), 
Tongji University, Shanghai 200092, China}
\email{zhangzili@tongji.edu.cn}

\begin{abstract}
    Let $f:S\to C$ be a proper surjective morphism from a smooth K\"ahler surface to a smooth curve.  We show that the local perverse filtration associated with the induced map $S^{[n]}\to C^{(n)}$ is multiplicative on each fiber if and only if $f$ is an elliptic fibration.
\end{abstract}

\maketitle
\tableofcontents

\section{Introduction}
\subsection{Multiplicativity of perverse filtrations and $P=W$}
Throughout this paper, all varieties are over complex number $\BC$ and all cohomology groups are of rational coefficients. 

For a K\"ahler manifold $Y$, the perverse truncation functors $^\Fp\tau_{\le k}$ in the perverse $t$--structure on the bounded derived category $D^b_c(Y)$ of constructible sheaves (see \cite{BBD}) give natural morphisms
\[
^\Fp\tau_{\le k}K\to K,
\]
where $K\in D^b_c(Y)$. Let $f:X\to Y$ be a proper surjective morphism of K\"ahler manifolds with equidimensional fibers. Following \cite{dCHM}, we define
\[
P_kH^*(X)=\textup{Im}\left\{\BH^*(Y,{^\Fp\tau_{\le k+\dim Y}}Rf_*\BQ_X)\to H^*(X)\right\},
\]
and call the increasing filtration
\[
P_0H^*(X)\subset P_1H^*(X)\subset \cdots\subset H^*(X)
\]
\emph{the perverse filtration associated with} $f$. A perverse filtration is called \emph{multiplicative} if 
\begin{equation} \label{0000}
P_kH^d(X)\times P_lH^e(X)\xrightarrow{\cup}P_{k+l}H^{d+e}(X),~~~~k,l,d,e\ge0.
\end{equation}

The perverse filtration associated with a general morphism is \emph{not} always multiplicative, even for proper morphism between smooth algebraic varieties; see \cite[Example 1.5]{Z}. The multiplicativity of perverse filtration (\ref{0000}) is closely related to the $P=W$ phenomenon, which is first observed by de Cataldo, Hausel and Migliorini \cite{dCHM} in the study of Hitchin moduli spaces. Let $C$ be a smooth projective curve of genus $g(C)\ge2$, and let $G$ be a reductive group. Then there is a canonical homoemorphism, called the nonabelian Hodge correspondence, between the moduli space $M_D$ of semistable $G$-Higgs bundles on $C$ and the moduli space $M_B$ of semisimple $G$-representations of the fundamental group $\pi_1(C)$; see \cite{S}. The $P=W$ conjecture asserts that under the identification
\[
H^*(M_D,\BQ)=H^*(M_B,\BQ)
\]
induced by the nonabelian Hodge correspondence, the perverse filtration on $M_D$ associated with the Hitchin fibration matches the Hodge-theoretic weight filtration on $M_B$, \emph{i.e.}
\[
P_kH^*(M_D,\BQ)=W_{2k}H^*(M_B,\BQ)=W_{2k+1}H^*(M_B,\BQ),~~k\ge0.
\]

The $G=\textup{GL}_n$ cases are proved recently by \cite{MS} and \cite{HMMS} independently, and the $\textup{SL}_p$ case for prime number $p$ is proved by \cite{dCMS2}. Since the Hodge-theoretic weight filtration is always multiplicative, the $P=W$ identity forces the perverse filtration to be multiplicative. In fact, the multiplicativity of the perverse filtrations associated with the Hitchin fibrations is the key step in the proofs. 
The $P=W$ type identities are also found in other situations, see \cite{DMV,H,HLSY,Z2,Z4}.

\subsection{Local multiplicativity of Hitchin-type morphisms}
Let $f:X\to Y$ be a proper surjective morphism between K\"ahler manifolds with equidimensional fibers, and let $p$ be a closed point on $Y$. We define an increasing filtration of the cohomology of $F=f^{-1}(p)$, called \emph{the local perverse filtration at} $p$ (or on the fiber $F$), as 
\begin{equation}
P_kH^*(F)=\textup{Im}\{({^\Fp\tau}_{\le k+\dim Y}Rf_*\BQ_X)_p\to (Rf_*\BQ_X)_p\}.    
\end{equation}
Similar to (\ref{0000}), we say the local perverse filtration is \emph{multiplicative at $p$} if 
\begin{equation} 
P_kH^d(F)\times P_lH^e(F)\xrightarrow{\cup}P_{k+l}H^{d+e}(F),~~~~k,l,d,e\ge0.
\end{equation}

Although local perverse filtrations are defined on the cohomology groups of the fibers of the fibration, they do depend on the topology of an analytic neighborhood of the fibers. By the local nature of the perverse sheaves, we make the following folklore conjecture:

\begin{conj}
Let $f:X\to Y$ be a proper flat morphism between smooth algebraic varieties. The perverse filtration associated with $f$ is multiplicative if and only if the local perverse filtrations on all fibers are multiplicative. 
\end{conj}

In \cite{Z1}, we showed that for Hitchin-type fibration constructed via Hilbert schemes of points on fibered surfaces, the associated perverse filtration is multiplicative if and only if the fibered surface is an elliptic fibration. In this paper, we will study perverse filtrations on every fibers of such fibrations.

Let $f:S\to C$ be a proper surjective map from a smooth K\"ahler surface to a connected smooth complex curve with connected fibers. Let $S^{[n]}$ be the Douady space of $n$ points on $S$, \emph{i.e.} the moduli of 0-dimensional analytic subvarieties of length $n$ on $S$. 
By taking the support we have a morphism $\pi_n:S^{[n]}\to S^{(n)}$, called the Douady-Barlet morphism, where $S^{(n)}$ is the symmetric product of $n$ points on $S$. Let $f^{[n]}:S^{[n]}\to C^{(n)}$ be the composition of $\pi_n$ and the natural morphism $S^{(n)}\to C^{(n)}$. Let $\nu=(\nu_1,\cdots,\nu_l)$ be a partition of $n$, \emph{i.e.} $\nu_1+\cdots+\nu_l=n$ and $\nu_1\ge\cdots\ge\nu_l>0$. Let $\textit{\textbf{x}}=\nu_1x_1+\cdots+\nu_lx_l$ be a point in $C^{(n)}$, where $x_1,\cdots,x_l$ are distinct points on $C$. Denote the diagonal $\Delta=\{\textit{\textbf{x}}\in C^{(n)}|\nu_1\ge2\}$ where at least two points on $C$ collide. We have

\begin{thm}[Theorem \ref{4.3}]
Let $n\ge 2$, $f:S\to C$, and $f^{[n]}:S^{[n]}\to C^{(n)}$  as above. 
\begin{enumerate}
    \item If $\textbf{x}\not\in \Delta$, then the local perverse filtration at $\textbf{x}$ is multiplicative.
    \item If $\textbf{x}\in\Delta$, then the local perverse filtration at $\textbf{x}$ is multiplicative if and only if $f$ is an elliptic fibration.
\end{enumerate}
\end{thm}

Since being an elliptic fibration is a global property, the multiplicativity of the local perverse filtrations on all the fibers of $f^{[n]}$ can be checked on one of them. More precisely, we have

\begin{thm} [Theorem \ref{4.4}]
Let $n\ge2$, $f:S\to C$ and $f^{[n]}:S^{[n]}\to C^{(n)}$ as above. The following are equivalent.
\begin{enumerate}
\item The morphism $f$ is an elliptic fibration.
\item The local perverse filtration at some point $\textbf{x}\in\Delta$ is multiplicative.
\item The local perverse filtration at all points in $C^{(n)}$ are multiplicative.
\end{enumerate}
In particular, if the perverse filtration associated with $f^{[n]}$ is multiplicative, then the local perverse filtration at all points are multiplicative. When $f$ is a local model in the sense of Theorem \ref{thm0}, then the converse is true.
\end{thm}

As an application, we provide a systematic framework to study the topology of the fibers of Hitchin fibrations described in Boalch's conjecture. In \cite[Remark 11.3]{B}, Boalch conjectures that if $M$ is a smooth 2 dimensional moduli space of (possibly wildly parabolic) Higgs bundles, then the Hilbert scheme $M_D:=M^{[n]}$ is also a moduli space of Higgs bundles.  Boalch's conjecture is verified for five families of moduli spaces of mild parabolic Higgs bundles in \cite{G}, and the corresponding $P=W$ conjecture are proved in \cite{SZ,Z}. The general case is still open. The understanding of the topology of the Hitchin fibers, which are a priori the compactified Jacobian of spectral curves, may be helpful to find the correct moduli description of $M^{[n]}$ predicted by Boalch's conjecture.

\subsection{Outline}
This paper is organized as follows. In Section 2, after recalling some basic properties of perverse filtrations, we define local perverse filtration on the fibers and develop its properties. In Section 3, we study the cohomology ring of Douady spaces of points on surfaces and  show that the main result of \cite{Z1} holds in the complex analytic setting. In Section 4, we realize fibers $(f^{[n]})^{-1}(\textbf{\textit{x}})$ as a product of the central fibers of the local models. Moreover, such a product decomposition is compatible with the perverse filtrations, and hence proves the main theorem. 

\subsection{Acknowledgements}
I thank Mark de Cataldo and Junliang Shen for many useful conversations.

\section{Perverse filtrations for K\"ahler manifolds}
\subsection{Cohomology and compactly supported cohomology}
Let $Y$ be a K\"ahler manifold. Let $D^b_c(Y)$ be the bounded derived category of constructible sheaves of $\BQ$-vector spaces on $Y$. For any object $\CC\in D^b_c(Y)$, the perverse truncation functor $^\Fp\tau_{\le k}$ defined by the perverse $t$-structure on $D^b_c(Y)$ induces a natural morphism
\[
^\Fp\tau_{\le k}\CC\to\CC.
\]

Let $f:X\to Y$ be a proper surjective morphism between K\"ahler manifolds. Let $r(f)=\dim X\times_YX-\dim X$ be the defect of semismallness of $f$. The morphism
\[
^\Fp\tau_{\le k}(Rf_*\BQ_X[\dim X])\to Rf_*\BQ_X[\dim X]
\]
induces a morphism in cohomology groups
\[
\BH^d\left(^\Fp\tau_{\le k}(Rf_*\BQ_X[\dim X])\right)\to \BH^d(Rf_*\BQ_X[\dim X])=H^{d+\dim X}(X,\BQ).
\]
We define the $P_kH^d(X)$ by
\[
P_kH^d(X):=\textrm{Im}\left\{\BH^{d-\dim X}\left(^\Fp\tau_{\le k-r(f)}(Rf_*\BQ_X[\dim X])\right)\to H^{d}(X,\BQ)\right\},
\]
and call the increasing filtration
\begin{equation}\label{2}
P_0H^d(X)\subset P_1H^d(X)\subset\cdots\subset H^d(X),~~~~d\ge0
\end{equation}
the perverse filtration associated with the map $f$. The shifts of the indices in the definition guarantee that the perverse filtration $P_\bullet H^*(X)$ concentrates at interval $[0,2r(f)]$. When $f$ is a fibration with equidimensional fiber, we have $r(f)=\dim X-\dim Y$, and hence
\[
P_kH^d(X)=\textrm{Im}\left\{\BH^{d}\left(Y,{^\Fp\tau}_{\le k+\dim Y}Rf_*\BQ_X\right)\to H^{d}(X,\BQ)\right\}.
\]
We say that the perverse filtration (\ref{2}) is \emph{multiplicative} if
\[
P_kH^d(X)\cup P_lH^e(X)\subset P_{k+l}H^{d+e}(X),   ~~~~\forall k,l,d,e\ge0
\]
with respect to the cup product. A class $\alpha\in H^d(X,\BQ)$ is called \emph{of perversity} $k$, denoted as $\Fp(\alpha)=k$ if $\alpha\in P_kH^d(X)$ and $\alpha\not\in P_{k-1}H^d(X)$. A basis $\{\beta_1,\cdots,\beta_N\}$ of the total cohomology group $H^*(X)$ is called a \emph{filtered basis} if
\[
P_kH^d(X)=\left\langle\beta_i|\beta_i\in P_kH^d(X),1\le i\le N\right\rangle.
\]

Similarly, we define the perverse filtration on the compactly supported cohomology $H^*_c(Y)$ associated with $f:X\to Y$ as
\[
P_kH^d_c(X):=\textrm{Im}\left\{\BH^{d-\dim X}_c\left(^\Fp\tau_{\le k-r(f)}(Rf_*\BQ_X[\dim X])\right) \to H^d_c(X,\BQ)\right\}.
\]
It follows from the definition that the forgetful map $H^d_c(X)\to H^d(X)$ preserves the perverse filtrations, \emph{i.e.}
\[
\begin{tikzcd}
P_kH_c^d(X)\arrow[r]\arrow[d,hook]& P_kH^d(X)\arrow[d,hook]\\
H_c^d(X)\arrow[r]& H^d(X)
\end{tikzcd}
\]
commutes.

\subsection{Perverse decompositions}
Let $f:X\to Y$ be a proper surjective map between K\"ahler manifolds. If we have a decomposition
\[
Rf_*\BQ_X[\dim X-r(f)]=\bigoplus_{i=0}^{2r(f)}\CP_i[-i],
\]
where the $\CP_i$ are perverse sheaves, then we have a decomposition
\[
H^d(X)=\bigoplus_{i=0}^{2r(f)}\BH^d(\CP_i[-i]).
\]
Any decomposition 
\begin{equation}\label{3}
H^*(X)=\bigoplus_{i=0}^{2r(f)}G_iH^*(X)
\end{equation}
obtained in this way is called a \emph{perverse decomposition associated with $f$}. It follows that
\begin{equation}\label{03}
P_kH^d(X)=\bigoplus_{i=0}^kG_iH^d(X).
\end{equation}
A perverse decomposition (\ref{3}) is called \emph{strongly multiplicative} if
\[
G_kH^d(X)\cup G_l H^e(X)\subset G_{k+l}H^{d+e}(X).
\]
It follows directly from the definition that if there exists a strongly multiplicative perverse decomposition associated with $f:X\to Y$, then the perverse filtration is multiplicative. But the inverse is \emph{not} true in general. If a non-zero class $\alpha\in G_kH^d(X)$, we say $\Fg(\alpha)=k$. A basis $\{\beta_1,\cdots,\beta_N\}$ of $H^*(X)$ is called \emph{adapted to the perverse decomposition} $G_\bullet H^*(X)$ if
\[
G_kH^d(X)=\left\langle\beta_i\mid\beta_i\in G_kH^d(X),1\le i\le N\right\rangle.
\]
A basis which is adapted to any perverse decomposition is a filtered basis.

The forgetful map $H^d_c(X)\to H^d(X)$ preserves any perverse decomposition associated with $f$, \emph{i.e.}
\[
G_iH^d_c(X)\to G_iH^d(X).
\]

Perverse decompositions on cohomology and on compactly supported cohomology behave well under the Poinar\'e duality. In fact, since $Rf_*\BQ_X[\dim X]$ is self-dual with respect to the Verdier duality,  $\CP_i\cong\BD\CP_{2r(f)-i}$, where $\BD$ is the Verdier duality. Therefore, under the Poicar\'e duality
\[
H^d(X)\cong H^{2\dim X-d}_c(X)^\vee,~~~~d\ge 0,
\]
the direct summands of the perverse decomposition (\ref{3}) are identified as
\[
G_iH^d(X)\cong \left(G_{2r(f)-i}H^{2\dim X-d}_c(X)\right)^\vee, ~~~~ i,d\ge 0.
\]

In other word, the perverse filtration (\ref{3}) associated with $f$ is a $G$-decomposition in the sense of \cite[Definition 2.8]{Z1}. More precisely, we have the following counterpart of \cite[Proposition 4.1]{Z1}.

\begin{prop}\label{2.1}
Let $f:X\to Y$ be a proper surjective map between K\"ahler manifolds. Given a decomposition into shifted perverse sheaves
\begin{equation} \label {2.1.1}
Rf_*\BQ_X[\dim X-r(f)]=\bigoplus_{i=0}^{2r(f)}\CP_i[-i],
\end{equation}
where $\CP_i$ are perverse sheaves. Then for any basis $\{\beta_1,\cdots,\beta_N\}$ of $H^*(X)$ adapted to the perverse decomposition (\ref{2.1.1}), the dual basis $\{\beta^1,\cdots,\beta^N\}$ of $H^*_c(X)$ in the sense of Poincar\'e paring is also adapted to the perverse decomposition (\ref{2.1.1}), and moreover, we have $\Fg(\beta_i)+\Fg(\beta^i)=2r(f)$.
\end{prop}

Let 
\begin{equation}\label{Delta}
\Delta_*:H^*(X)\to H^*_c(X)\otimes H^*(X)
\end{equation}
be the $\otimes$--Hom adjoint of the cup product 
\[
\cup:H^*(X)\otimes H^*(X)\to H^*(X).
\]
Then 
\[
\Delta_*(\alpha)=\sum_{i=1}^N\beta^i\otimes\beta_i\alpha
\]
\begin{cor}\label{2.2}
Let $f:X\to Y$ be a proper surjective map between K\"ahler manifolds. Fix a decomposition
\[
Rf_*\BQ_X[\dim X-r(f)]=\bigoplus_{i=0}^{2r(f)}\CP_i[-i],
\]
Let $\{\beta_1,\cdots,\beta_N\}$ of $H^*(X)$ be a basis of $H^*(X)$ adapted to the perverse decomposition $G_\bullet H^*(X)$. Let $\{\beta^1,\cdots,\beta^N\}$ be the dual basis in $H^*_c(X)$. Then  
\[
\Delta_*(1)=\sum_{i=1}^N\beta^i\otimes\beta_i
\]
such that
\[
\Fg(\beta_i)+\Fg(\beta^i)=2r(f)
\]
for all $1\le i\le N$.
\end{cor}

\begin{rmk}
When $X$ is smooth a projective variety, we have $H^*_c(X)=H^*(X)$ and that $\Delta_*$ is the push-forward along the diagonal embedding. Proposition \ref{2.1} and Corollary \ref{2.2} recover \cite[Propositon 3.1]{Z} and \cite[Proposition 3.8]{Z}, respectively. 
\end{rmk}

\subsection{Local perverse filtrations on fibers}
Let $f:X\to Y$ be a proper morphism between K\"ahler manifolds, and let $\iota:\{p\}\to Y$ be a closed point on $Y$. 
\[
\begin{tikzcd}
F\arrow[r,"j"]\arrow[d,"g"]& X\arrow[d,"f"]\\
\{p\}\arrow[r,"\iota"]& Y
\end{tikzcd}
\]
We define a filtration of the cohomology of $F=f^{-1}(p)$, called the local perverse filtration at point $p$ (or on fiber $F$), as follows. Starting from the natural morphism
\[
{^\Fp\tau}_{\le k}(Rf_*\BQ_X[\dim X])\to Rf_*\BQ_X[\dim X],
\]
we have
\[
\iota^*{^\Fp\tau}_{\le k}(Rf_*\BQ_X[\dim X])\to \iota^*Rf_*\BQ_X[\dim X].    
\]
By the proper base change theorem, 
\begin{equation}
\iota^*Rf_*\BQ_X[\dim X]\cong Rg_*j^*\BQ_X[\dim X]=Rg_*\BQ_F[\dim X].
\end{equation}
So there is a natural map of hypercohomology groups
\begin{equation}
    \BH^d\left(\iota^*{^\Fp\tau_{\le k}}(Rf_*\BQ_X[\dim X])\right)\to \BH^d(Rg_*\BQ_F[\dim X])=H^{d+\dim X}(F,\BQ).
\end{equation}

\begin{defn}
Let $f:X\to Y$ be a proper morphism between K\"ahler manifolds. Let $\iota:\{p\}\to Y$ be a point on $Y$, and let $F:=f^{-1}(p)$ be the fiber over $p$. Define
\begin{equation} \label{fiber}
P_kH^d(F):=\textrm{Im} \left\{\BH^{d-\dim X}(\iota^*{^\Fp\tau_{\le {k-r(f)}}}Rf_*\BQ_X[\dim X])\to H^d(F,\BQ)\right\}.
\end{equation}
The increasing filtration
\[
P_0H^d(F)\subset P_1H^d(F)\cdots\subset H^d(F) 
\]
is called the local perverse filtration at $p$  (or on the fiber $F$) associated with the map $f$. When $f$ is a fibration with equidimensional fibers, the defect $r(f)$ equals the fiber dimension. In this case, the local perverse filtration $P_kH^*(F)$ at any fiber $F$ is concentrated at $0\le k \le 2\dim F$, and
\[
P_kH^*(F)=\textup{Im}\{({^\Fp\tau}_{\le k+\dim Y}Rf_*\BQ_X)_p\to (Rf_*\BQ_X)_p\}.    
\]
We say that the local perverse filtration at $p$ associated with $f$ is multiplicative, or simply \emph{locally multiplicative at $p$},  if the cup product
\[
\cup_F:H^*(F)\otimes H^*(F)\to H^*(F)
\]
satisfies 
\[
P_kH^i(F)\cup_F P_lH^j(F)\subset P_{k+l}H^{i+j}(F), ~~\forall i,j,k,l\ge0.
\]
The perverse filtration associated with $f$ is called \emph{locally multiplicative} if it is locally multiplicative at all points on $Y$.
\end{defn}

\begin{rmk}
\begin{enumerate}
    \item If we have a perverse decomposition
    \[
    Rf_*\BQ_X[\dim X-r(f)]\cong \bigoplus_{i=0}^{2\dim F}\CP_i[-i]
    \]
    where $\CP_i$ are perverse sheaves, then the perverse filtration on fiber $F$ is computed by
    \begin{equation}\label{abc}
    P_kH^*(F)\cong\bigoplus_{i=0}^{k}\BH^*(\{p\},\iota^*\CP_i).
    \end{equation}
    \item The local perverse filtration on fiber $F=f^{-1}(p)$ associated with the map $f:X\to Y$ is different from the perverse filtration associated with $F\to \{p\}$.  
    \item The local perverse filtration on fiber $F$ associated with the map $f$ is \emph{not} induced by $P_\bullet H^*(X)$ via the natural restriction $H^*(X)\to H^*(F)$, especially when the restriction map is not surjective. 
\end{enumerate}
\end{rmk}

The following proposition implies the local nature of the local perverse filtration on fibers.

\begin{prop} \label{2.6}
Let $f:X\to Y$ be a proper surjective morphism between K\"ahler manifolds. Let $p$ be a point on $Y$, and let $F=f^{-1}(p)$ be the fiber over $p$. Then there exists a neighborhood $U$ of $p$, such that the 
restriction map $H^*(f^{-1}(U))\to H^*(F)$ is an isomorphism. For such $U$, the perverse filtration associated with $f|_{U}: f^{-1}(U)\to U$ restricts isomorphically to the local perverse filtration on $F$, i.e.
\[
P_kH^*(f^{-1}(U))\xlongrightarrow{\sim} P_kH^*(F), ~~\forall k\ge0.
\]
\end{prop}

\begin{proof}
For $d\ge0$, the higher pushforward sheaf $R^df_*\BQ_X$ is the sheafification of the presheaf
\[
V\mapsto H^d(f^{-1}(V),\BQ),
\]
so $(R^df_*\BQ_X)_p=\varinjlim H^d(f^{-1}(V))$. Since $R^df_*\BQ_X$ is a constructible sheaf, there exists a neighborhood $U$ of $p$ computing the stalk, \emph{i.e.} 
\[
(R^df_*\BQ_X)_p=H^d(f^{-1}(U),\BQ).
\]
By the proper base change theorem, 
\[
(R^df_*\BQ_X)_p\xlongrightarrow{\sim}H^d(F,\BQ).
\]
Therefore, the restriction map
\begin{equation}\label{rest0}
H^d(f^{-1}(U),\BQ)\rightarrow H^d(F,\BQ),~~\forall d\ge0
\end{equation}
is an isomorphism. To see that this isomorphism preserves perverse filtrations, we fix a perverse decomposition 
\begin{equation} \label{rest}
Rf_*\BQ_{f^{-1}(U)}[\dim X-r(f)]=\bigoplus_{i=0}^{2\dim F}\CP_{i}[-i].
\end{equation}
Then the restriction map (\ref{rest0}), or equivalently
\[
\BH^{*}(U,Rf_*\BQ_{f^{-1}(U)})\to \BH^{*}(\{p\},\iota^*Rf_*\BQ_{f^{-1}(U)})
\]
is decomposed as the direct sum of
\begin{equation}\label{rest1}
\BH^{*}(U,\CP_i)\to \BH^{*}(\{p\},\iota^*\CP_i).
\end{equation}
Therefore, (\ref{rest1}) is an isomorphism for each $i$, and hence 
\[
\bigoplus_{i=0}^k\BH^{*}(U,\CP_i)\xlongrightarrow{\sim} \bigoplus_{i=0}^k\BH^{*}(\{p\}, \iota^*\CP_i),
\]
\emph{i.e.}
\[
P_kH^*(f^{-1}(U))\xlongrightarrow{\sim} P_kH^*(F), ~~\forall k\ge0.
\]
\end{proof}

Since restriction map in cohomology is map of algebras, multiplicativity of the local perverse filtration on the fiber $F$ is equivalent to the multiplicativity of the perverse filtration associated with $f:f^{-1}(U)\to U$ for sufficiently small analytic neighborhood $U$. That's why we call it local multiplicativity.

\subsection{Functoriality of local perverse filtrations on fibers}
In this section we prove some functoriality results on local perverse filtration on fibers. To simplify notation, we denote $F_f(y)$ the fiber over $y$ of a given map $f:X\to Y$, \emph{i.e.} $F_f(y)=f^{-1}(y)$. 
\begin{prop}\label{2.7}
Let $f_i:X_i\to Y_i$, $1\le i\le n$ be proper morphisms of K\"ahler manifolds. Let $y_i\in Y_i$ be closed points, and $\textbf{\textit{y}}=(y_1,\cdots,y_n)\in Y_1\times\cdots\times Y_n$. Let 
$f:X_1\times\cdots\times X_n\to Y_1\times\cdots\times Y_n$ be the product. Then under the isomorphism
\[
F_{f}(\textbf{\textit{y}})=F_{f_1}(y_1)\times \cdots\times F_{f_n}(y_n),
\]
the local perverse filtrations on the fibers satisfy the K\"unneth formula, i.e.
\[
P_kH^*(F_{f}(\textbf{\textit{y}}))=\left\langle \alpha_1\boxtimes\cdots\boxtimes\alpha_n\mid \alpha_i\in P_{k_i}H^*(F_{f_i}(y_i)),\,\sum k_i=k\right\rangle.
\]
In particular, the local perverse filtration on the fiber $F_{f}(\textbf{\textit{y}})$ is multiplicative if and only if the local perverse filtrations on the fibers $F_{f_i}(y_i)$ are all multiplicative.
\end{prop}

\begin{proof}
The K\"unneth formula follows from taking the stalk at $\textbf{\textit{y}}$ and induction in \cite[Proposition 2.1]{Z}. To see the equivalence of multiplicativity of perverse filtrations, we first note that the inclusion
\begin{align*}
H^*(F_{f_i}(y_i))&\hookrightarrow \bigotimes_{i=1}^n H^*(F_{f_i}(y_i))=H^*(F_f(\textbf{\textit{y}}))\\
\gamma_i&\mapsto 1\otimes\cdots\otimes\gamma_i\otimes\cdots\otimes1
\end{align*}
is a subsuperalgebra which preserves the local perverse filtration. So if $P_\bullet H^*(F_f(\textbf{\textit{y}}))$ is multiplicative, then $P_\bullet H^*(F_{f_i}(y_i))$ are all multiplicative. Conversely, if all the local perverse filtrations $P_\bullet H^*(F_{f_i}(y_i))$ are all multiplicative, $P_\bullet H^*(F_f(\textbf{\textit{y}}))$ is multiplicative by the K\"unneth formula.
\end{proof}

For a complex manifold $X$, Let $X^{(n)}=X^{\times n}/\mathfrak{S}_n$ be the $n$-fold symmetric product of $X$. Elements in $X^{(n)}$ are written additively as $\textbf{\textit{x}}=x_1+\cdots+x_n$, where $x_i\in X$ for $1\le i\le n$. For any map $f:X\to Y$, the $\mathfrak{S}_n$-equivariant quotient of $f^n:X^{\times n}\to Y^{\times n}$ yields a map $f^{(n)}:X^{(n)}\to Y^{(n)}$. For simplicity, for $\textbf{\textit{y}}\in Y^{(n)}$ we denote 
\[
F^{(n)}(\textbf{\textit{y}}):=F_{f^{(n)}}(\textbf{\textit{y}}).
\]

\begin{prop} \label{2.8}
Let $f:X\to Y$ be a proper map between K\"ahler manifolds. Let $f^{(n)}:X^{(n)}\to Y^{(n)}$ be the induced map. Let $\textbf{y}$ be a point in $Y^{(n)}$.
\begin{enumerate}
    \item If $\textbf{y}=ny$ for some $y\in Y$, the cohomology group
    \[
    H^*(F^{(n)}(ny))=\textup{Sym}^n H^*(F_f(y))\]
    is the $n$-th super-symmetric power, and the local perverse filtration on $H^*(F^{(n)}(ny))$ is obtained as the $\mathfrak{S}_n$-descent of the perverse filtration on $H^*(F_{f^n}(y,\cdots,y))$.
    \item For a general partition $\nu=(\nu_1,\dots,\nu_l)$ of $n$, and a point $\textbf{y}=\nu_1y_1+\cdots+\nu_ly_l$  in $Y^{(n)}$ with distinct $y_i$'s, the fiber over $\textbf{y}$ is a product
    \[
    F^{(n)}(\nu_1y_1+\cdots+\nu_ly_l)=F^{(\nu_1)}(\nu_1y_1)\times\cdots\times F^{(\nu_l)}(\nu_ly_l)
    \]
    and the local perverse filtrations satisfy the K\"unneth formula, where the local perverse filtration on $H^*(F^{(\nu_i)}(\nu_iy_i))$ is associated with $f^{(\nu_i)}$. 
\end{enumerate}
\end{prop}

\begin{proof}
(1) follows from \cite[Lemma 2.9]{Z} by taking the stalk at $ny$. To see (2), let $\{U_i\}_{i=1}^l$ be disjoint analytic neighborhoods of $y_i$ in $Y$. Then the symmetric powers $U_i^{(\nu_i)}$ are analytic neighborhoods of $\nu_iy_i$ in $Y^{(\nu_i)}$. Since $U_i$ are disjoint, $U_1^{(\nu_1)}\times\cdots\times U_l^{(\nu_l)}$ maps isomorphically to its image under
\begin{align*}
    r_{\nu,Y}:Y^{(\nu_1)}\times\cdots \times Y^{(\nu_l)}&\to Y^{(n)}\\
    \left(\sum_{i=1}^{\nu_1}y_{1i},\cdots,\sum_{i=1}^{\nu_l}y_{li}\right)& \mapsto\sum_{i,j}y_{ij}.
\end{align*}
Therefore, $U=U_1^{(\nu_1)}\times\cdots\times U_l^{(\nu_l)}$ is an analytic neighborhood of $\nu_1y_1+\cdots+\nu_ly_l$. Since
\[
\begin{tikzcd}
X^{(\nu_1)}\times\cdots\times X^{(\nu_l)}\arrow[r,"r_{\nu,X}"]\arrow[d]& X^{(n)}\arrow[d,"f^{(n)}"]\\
Y^{(\nu_1)}\times\cdots\times Y^{(\nu_l)}\arrow[r,"r_{\nu,Y}"]& Y^{(n)}
\end{tikzcd}
\]
is a Cartesian product, $\left(f^{(n)}\right)^{-1}U\to U$ is the product of maps 
\[
\left(f^{(\nu_i)}\right)^{-1}\left(U_i^{(\nu_i)}\right)\to U_i^{(\nu_i)}.
\]
By Proposition \ref{2.7}, (2) follows. 
\end{proof}

\section{Douady Spaces}
To study the analytically local behavior of singular fibers of a Hitchin-type fibration of Hilbert schemes, we need to introduce the complex analytic counterpart of Hilbert schemes, called the Douady spaces.
Let $S$ be a smooth complex surface.  The Douady space of length $n$ on surface $S$, denoted as $S^{[n]}$, parametrizes 0-dimensional analytic subvarieties of length $n$ on $S$. $S^{[n]}$ is a smooth complex manifold of dimension $2n$ and there is a natural projective morphism $\pi:S^{[n]}\to S^{(n)}$ to the symmetric product, sending a $0$-dimensional subvariety to its support, called the Douady-Barlet map. When $S$ is a smooth quasi-projective algebraic surface, the Douady-Barlet map coincides with the Hilbert-Chow morphism. We refer to \cite[Section 2]{dCM} for details.

The cohomology groups of Douady spaces are studied in \cite{dCM} by analyzing the local structure of the Douady-Barlet map. Let $\nu=1^{a_1}\cdots n^{a_n}$ be a partition of $n$, \emph{i.e.} $i$ appears $a_i$ times. Let 
\begin{equation} \label{999}
S^{(\nu)}=S^{(a_1)}\times\cdots\times S^{(a_n)}
\end{equation}
be the product of symmetric products. Then geometric objects on $S$ naturally define objects of the same type on $S^{(\nu)}$ via products and quotients. Following \cite[Definition 2.5]{Z}, for $A\in D^b_c(S)$ we denote
\begin{equation} \label{1210}
A^{(\nu)}=\boxtimes_{i=1}^n \left(A^{\boxtimes a_i}\right)^{\mathfrak{S}_{a_i}}\in D^b_c(S^{(\nu)}).
\end{equation}
We have
\begin{equation} \label{1211}
H^*(S^{(\nu)},A^{(\nu)})=\bigotimes_{i=1}^n H^*(S^{(a_i)},A^{(a_i)})=
\bigotimes_{i=1}^n \textup{Sym}^{a_i}H^*(S,A).
\end{equation}
Let 
\begin{align*}
    \iota_\nu: S^{(\nu)}=S^{(a_1)}\times\cdots\times S^{(a_n)}&\to S^{(n)}\\    (\textbf{\textit{x}}_1,\cdots,\textbf{\textit{x}}_n)&\mapsto(\textbf{\textit{x}}_1+2\textbf{\textit{x}}_2+\cdots+n\textbf{\textit{x}}_n).
\end{align*}

\begin{prop}{\cite[Theorem 4.1.1]{dCM}}\label{3.1}
Let $S$ be a smooth complex surface. Let $\pi:S^{[n]}\to S^{(n)}$ be the Douady-Barlet map. Then there is a canonical decomposition
\[
R\pi_*\BQ_{S^{[n]}}[2n]=\bigoplus_{\nu\vdash n} \iota_{\nu*} \BQ_{S^{(\nu)}}[2l(\nu)],
\]
where $\nu$ runs through all partitions of $n$ and $l(\nu)$ is the length of $\nu$.
In particular, the G\"ottsche formula holds.
\[
H^*(S^{[n]})[2n]=\bigoplus_{\nu=1^{a_1}\cdots n^{a_n}} \bigotimes_{i=1}^n H^*(S^{(a_i)})[2(a_1+\cdots+a_n)].
\]
\end{prop}

When $S$ is a smooth K\"ahler surface equipped with a proper holomorphic morphism $f:S\to C$ onto a smooth curve, there is a proper morphism $f^{[n]}:S^{[n]}\to S^{(n)}\to C^{(n)}$, as shown below. 
\[
\begin{tikzcd}
 & S^{[n]}\arrow[d,swap,"\pi"]\arrow[dd,bend left,"f^{[n]}"]\\
S^{(\nu)}\arrow[r,"\iota_\nu"]\arrow[d,swap,"f^{(\nu)}"]& S^{(n)}\arrow[d,swap,"f^{(n)}"]\\
C^{(\nu)}\arrow[r,"\iota_\nu"]& C^{(n)}.\\
\end{tikzcd}
\]

The decomposition in Proposition \ref{3.1} can be pushforward further to $C^{(n)}$ and the perverse filtration on $H^*(S^{[n]})$ associated with $f^{[n]}$ is described as follows\footnote{This result was originally stated when $S$ and $C$ are quasi-projective, but the proof applies to the K\"ahler case. In fact, the quasi-projectivity was introduced to apply the BBDG perverse decomposition theorem, which is true for proper morphism of K\"ahler manifolds.}.

\begin{prop}{\cite[Proposition 4.12, Corollary 4.14]{Z}} \label{perv}
There is an isomorphism
\begin{equation}\label{1212}
Rf^{[n]}_*\BQ_{S^{[n]}}[2n]=\bigoplus_{\nu\vdash n} \iota_{\nu*}(Rf_*\BQ_S)^{(\nu)}[2l(\nu)].
\end{equation}
   Under the isomorphism 
\[
H^{d}(S^{[n]})=\bigoplus_{\nu} H^{d+2l(\nu)-2n}(S^{(\nu)}),
\]
the perverse filtration can be identified as
\[
P_kH^{d}(S^{[n]})=\bigoplus_{\nu} P_{k+l(\nu)-n}H^{d+2l(\nu)-2n}(S^{(\nu)}),
\]
where the perverse filtration on the right side is associated with the induced morphism $f^{(\nu)}:S^{(\nu)}\to C^{(\nu)}$.
\end{prop}

The results in \cite[Section 4]{Z1} for algebraic varieties hold in K\"ahler setting. In fact, they are based on two ingredients: the relation between the perverse filtration and the ring structure of $H^*(S^{[n]})$ studied in \cite[Section 4.2]{Z1}, and the geometry of $f:S\to C$ studied in \cite[Section 4.3]{Z1}. In the K\"ahler case, the cohomology group of $S^{[n]}$ and the perverse filtration associated with $f^{[n]}$ are described in Proposition \ref{3.1} and \ref{perv}, which are exactly the same as Hilbert schemes. Furthermore, the Nakajima operators and Virasoro operators for Douady spaces satisfy the same relations as the ones for Hilbert schemes; see \cite{N}. Therefore, the cohomology ring of Douady spaces satisfy the same formula as the main theorem of \cite{LQW}. It is straightforward to check the arguments \cite[Section 4.3]{Z1} does not depend on whether $C$ is algebraic or not. For our purpose, we need the following counterparts of \cite[Theorem 4.9, 4.17]{Z1} in the K\"ahler setting.

\begin{thm} \label{thm0}
Let $n\ge2$. Let $f:S\to C$ a proper holomorphic fibration of a smooth K\"ahler surface over a smooth curve. If the perverse filtration associated with $f^{[n]}$ is multiplicative, then $f$ is an elliptic fibration, i.e. general fibers have genus 1. When $f$ is a local model, i.e. the restriction map $H^*(S)\to H^*(f^{-1}(p))$ for some $p\in C$ is an isomorphism and $C$ is contractible, then the converse is also true.
\end{thm}

\section{Local multiplicativity of Hilbert schemes of fibered surfaces}
In this section we will show that the local perverse filtrations on the fibers of $f^{[n]}:S^{[n]}\to C^{(n)}$ are all multiplicative if and only if $f$ is an elliptic fibration. We analyze the fibers of $f^{[n]}$ first. To simplify the notation, let $F^{[n]}(\textbf{\textit{x}})$ be the fiber $\left(f^{[n]}\right)^{-1}(\textbf{\textit{x}})$ for $\textbf{\textit{x}}\in C^{(n)}$. As a fiber of $f^{[n]}:S^{[n]}\to C^{(n)}$, the cohomology $H^*(F^{[n]}(\textbf{\textit{x}}))$ carries the local perverse filtration (\ref{fiber}) associated with $f^{[n]}$.

\begin{prop} \label{5.1}
Let $f:S\to C$ be a proper holomorphic fibration of a K\"ahler surface over a smooth curve. Let $\textbf{x}=\nu_1x_1+\cdots +\nu_lx_l$ be a point on $C^{(n)}$. Then the rational map induced by the union of subschemes with disjoint supports
\begin{align*}
q:S^{[\nu_1]}\times\cdots\times S^{[\nu_l]}&\dashrightarrow S^{[n]}\\
(\xi_1,\cdots,\xi_l)&\longmapsto \xi_1\cup\cdots\cup\xi_l
\end{align*}
induces an isomorphism
\[
F^{[n]}(\textbf{x})=F^{[\nu_1]}(\nu_1x_1)\times\cdots\times F^{[\nu_l]}(\nu_lx_l),
\]
where $F^{[\nu_i]}(\nu_ix_i)$ is the fiber of the morphism $S^{[\nu_i]}\to C^{(\nu_i)}$ over the point $\nu_ix_i\in C^{(\nu_i)}$, $1\le i\le l$. Furthermore, the local perverse filtration on the fiber $F^{[n]}(\textbf{x})$ satisfies the K\"unneth formula in the sense of Proposition \ref{2.7}, i.e.
\[
P_kH^*\left(F^{[n]}(\textbf{x})\right)=\left\langle\gamma_1\boxtimes\cdots\boxtimes\gamma_l\mid \gamma_i\in P_{k_i}H^*\left(F^{[\nu_i]}(\nu_ix_i)\right),\sum_{i=1}^lk_i=k\right\rangle,
\]
where the local perverse filtration on $F^{[\nu_i]}(\nu_i x_i)$ is defined by (\ref{fiber}) as a fiber of $S^{[\nu_l]}\to C^{(\nu_i)}$. The local perverse filtration on $F^{[n]}(\textbf{x})$ is multiplicative if and only if the local perverse filtrations on $F^{[\nu_i]}(\nu_ix_i)$ are all multiplicative for $1\le i\le l$.
\end{prop}

\begin{proof}
Consider the diagram
\[
\begin{tikzcd}
S^{[\nu_1]}\times\cdots\times S^{[\nu_l]}\arrow[r,dashed,"q"]\ar[d,swap,"\pi_{\nu_1}\times\cdots\times\pi_{\nu_l}"]& S^{[n]}\ar[d,"\pi_n"]\\
S^{(\nu_1)}\times\cdots\times S^{(\nu_l)}\arrow[r,"r_{\nu,S}"]\arrow[d,swap,"f^{(\nu_1)}\times \cdots\times f^{(\nu_l)}"]&S^{(n)}\arrow[d,"f^{(n)}"]\\
C^{(\nu_1)}\times\cdots\times C^{(\nu_l)}\arrow[r,"r_{\nu,C}"]&C^{(n)},\\
\end{tikzcd}
\]
where $r_{\nu,S}$ and $r_{\nu,C}$ are defined in the proof of Proposition \ref{2.8}. Let 
\[
\CD=\{(\textit{\textbf{y}}_1,\cdots,\textit{\textbf{y}}_l)\mid \textit{\textbf{y}}_i\in C^{(\nu_i)}, 1\le i\le l \textup{ and } \exists \,i\neq j \textup{ such that } \textit{\textbf{y}}_i\cap \textit{\textbf{y}}_j\neq\varnothing\}
\]
be the closed locus in $C^{(\nu_1)}\times\cdots\times C^{(\nu_l)}$ where some points in different factors collide. Then the indeterminacy locus of $q$ 
\[
\left\{(\xi_1,\cdots,\xi_l)\mid\exists\, i\neq j\textrm{ such that Supp }\xi_i\cap \textrm{Supp }\xi_j\neq\varnothing\right\}
\]
is in the preimage $(f^{[\nu]})^{-1}(\CD)$, and the diagram
\[
\begin{tikzcd}
S^{[\nu_1]}\times\cdots\times S^{[\nu_l]}\setminus (f^{[\nu]})^{-1}\CD\arrow[r,"q"]\arrow[d,"f^{[\nu]}"]&S^{[n]}\arrow[d,"f^{[n]}"]\\
C^{(\nu_1)}\times\cdots C^{(\nu_l)}\setminus\CD\arrow[r,"r_C"]& C^{(n)}\\
\end{tikzcd}
\]
is a Cartesian diagram. By definition we have $r_{\nu,C}(\nu_1x_1,\cdots,\nu_lx_l)=\textbf{\textit{x}}$. The conclusion follows from the same argument as the proof of Proposition \ref{2.8}.(2).
\end{proof}

\begin{prop} \label{5.3}
Let $n\ge 2$. Let $f:S\to C$ be a proper holomophic fibration of a K\"ahler surface over a smooth curve. Then the local perverse filtration $P_\bullet H^*(F^{[n]}(nx))$ at $nx\in C^{(n)}$ is multiplicative if and only if $f$ is an elliptic fibration.
\end{prop}

\begin{proof}
Let $V$ be the analytic neighborhood of $x$ in $C$ obtained by Proposition \ref{2.6} with respect to the morphism $f:S\to C$. Let $T=f^{-1}(V)$ be the preimage, and we have a Cartesian diagram below. 

\[
\begin{tikzcd}
F\arrow[r,"\iota"]\arrow[d]&T\arrow[r]\arrow[d,"f_V"]&S\arrow[d,"f"]\\
\{x\}\arrow[r,"p"]&V\arrow[r]& C
\end{tikzcd}
\]
By the choice of $V$,  the adjunction morphism
\[
Rf_{V*}\BQ_T\to p_*p^*R_{V*}\BQ_T
\]
induces an isomorphism in cohomology
\begin{equation} \label{1313}
H^*(T)= H^*(V,Rf_{V*}\BQ_T)\xrightarrow{\cong} H^*(F).
\end{equation}

The constructions in Section 3 give a Cartesian diagram
\[
\begin{tikzcd}
F^{[n]}(nx)\arrow[r,"\iota^{[n]}"]\arrow[d]&T^{[n]}\arrow[d,"f^{[n]}_V"]\arrow[r] & S^{[n]}\arrow[d,"f^{[n]}"]\\
\{nx\}\arrow[r,"q"]&V^{(n)}\arrow[r]&C^{(n)}
\end{tikzcd}
\]
We claim that restriction map
\begin{equation}\label{1314}
\iota^{[n]*}:H^*(T^{[n]})\to H^*(F^{[n]}(nx))
\end{equation}
is an isomorphism. To prove the claim, we first note the following commutative diagram 
\[
\begin{tikzcd}
Rf^{[n]}_{V*}\BQ_{T^{[n]}}[2n]\arrow[r,"="]\arrow[d]&\bigoplus_\nu \iota_{\nu*}(Rf_*\BQ_T)^{(\nu)}[2l(\nu)]\arrow[d]\\
q_*q^*Rf^{[n]}_{V*}\BQ_{T^{[n]}}[2n]\arrow[r,"="]&q_*q^*\bigoplus_\nu \iota_{\nu*}(Rf_*\BQ_T)^{(\nu)}[2l(\nu)],
\end{tikzcd}
\]
where the horizontal morphisms are (\ref{1212}). Taking hypercohomology and using (\ref{1210}), we have
\[
\begin{tikzcd}
H^*\left(T^{[n]}\right)[2n]\arrow[r,"="]\arrow[d,"\iota^{[n]*}"]&\bigoplus_\nu \bigotimes_{i=1}^n \textup{Sym}^{a_i}H^*(T)[2l(\nu)]\arrow[d]\\
H^*\left(F^{[n]}(nx)\right)[2n]\arrow[r,"="]&\bigoplus_\nu \bigotimes_{i=1}^n \textup{Sym}^{a_i}H^*(F)[2l(\nu)].
\end{tikzcd}
\]

By (\ref{1313}), the vertical arrow on the right side is an isomorphism, and therefore the  claim (\ref{1314}) holds. Now we can apply Proposition \ref{2.6} to the fiber $F^{[n]}(nx)$ of the morphism $f_V^{[n]}$ to conclude that 
\[
P_\bullet H^*(T^{[n]})\xrightarrow{\sim} P_\bullet H^*(F^{[n]}(nx)).
\]
Therefore, the local perverse filtration on the fiber $F^{[n]}(nx)$ is multiplicative if and only if the perverse filtration associated with $f^{[n]}_V$ is multiplicative. By Theorem \ref{thm0}, the latter is equivalent to $f_V$ being an elliptic fibration. 
\end{proof}

Combining Proposition \ref{5.1} and \ref{5.3}, we have

\begin{thm} \label{4.3}
Let $f:S\to C$ be a proper holomorphic fibration of a K\"ahler surface over a smooth complex curve. Let $f^{[n]}:S^{[n]}\to C^{(n)}$ be the induced morphism. Let $\Delta\subset C^{(n)}$ be the locus where at least two points collide. The cohomology $H^*(F^{[n]}(\textbf{x}))$ is naturally endowed with the local perverse filtration (\ref{fiber}). Then:
\begin{enumerate}
    \item If $\textbf{\textit{x}}\not\in\Delta$, then the local perverse filtration at $\textbf{x}$ is multiplicative.
    \item If $\textbf{\textit{x}}\in\Delta$, then the local perverse filtration at $\textbf{x}$ is multiplicative if and only if $f$ is an elliptic fibration.
\end{enumerate}
\end{thm}

\begin{proof}
    Let $\nu=(\nu_1,\cdots,\nu_l)$ where $\nu_1\ge\nu_2\cdots\ge\nu_l\ge1$ be a partition of $n$ and let $\textbf{\textit{x}}=\nu_1x_1+\cdots+\nu_lx_l$ be a point in the symmetric product $C^{(n)}$, where $x_1,\cdots,x_l$ are distinct points in $C$. By Proposition \ref{5.1}, we have
    \[
    F^{[n]}(\textbf{\textit{x}})=F^{[\nu_1]}(\nu_1x_1)\times\cdots\times F^{[\nu_l]}(\nu_lx_l),
    \]
    and the local perverse filtration at $\textbf{\textit{x}}$ is multiplicative if and only if the local perverse filtrations at $\nu_ix_i\in C^{(\nu_i)}$ are all multiplicative. When $\nu_i=1$, the fiber $F^{[\nu_i]}(\nu_ix_i)$ is just $f^{-1}(x_i)$. By Proposition \ref{2.6} and \cite[Proposition 4.17]{Z}, the perverse filtration on an fiber of an surface properly fibered over a curve is always multiplicative. When $\nu_{i}\ge 2$, we see from Proposition \ref{5.3} that the multiplicativity of $F^{[\nu_{i}]}(\nu_{i}x_{i})$ is equivalent to $f$ is an elliptic fibration. Therefore, we conclude that when $\textbf{\textit{x}}\not\in\Delta$, the local perverse filtration at $\textit{\textbf{x}}$ always multiplicative, and when $\textbf{\textit{x}}\in\Delta$, the local perverse filtration at $\textit{\textbf{x}}$ is multiplicative if and only if $f$ is an elliptic fibration.  
\end{proof}

\begin{thm} \label{4.4}
Let $n\ge2$, $f:S\to C$ and $f^{[n]}:S^{[n]}\to C^{(n)}$ as above. The following are equivalent.
\begin{enumerate}
\item The morphism $f$ is an elliptic fibration.
\item The local perverse filtration at some point $\textbf{x}\in\Delta$ is multiplicative.
\item The local perverse filtration at all points in $C^{(n)}$ are multiplicative.
\end{enumerate}
In particular, if the perverse filtration associated with $f^{[n]}$ is multiplicative, then it is locally multiplicative. When $f$ is a local model in the sense of Theorem \ref{thm0}, then the converse is true.
\end{thm}

\begin{proof}
   It follows directly from Theorem \ref{4.3} that conditions (1), (2) and (3) are equivalent.
   The relation between the multiplicativity of perverse filtration associated with $f^{[n]}$ and the local multiplicativity follows from (1), (3) and Theorem \ref{thm0}. 
\end{proof}

\end{document}